\documentclass[12pt, english, a4paper]{amsart}

\usepackage{graphicx}

\usepackage{amsmath,amssymb}
\usepackage{bm}
\usepackage{euscript,color}
\usepackage{babel}
\usepackage{amsthm}
\usepackage{amsfonts}
\usepackage{latexsym}
\usepackage{fancyhdr}
\usepackage{enumerate}
 \usepackage{tikz-cd}
        \usepackage{amscd}

\usepackage[left=2.5cm, top=3cm, bottom=3cm, right=2.5cm]{geometry}

\newcommand{\R}{\mathbb{R}}
\newcommand{\C}{\mathbb{C}}

\newcommand{\s}{\mathbb{S}}

\newtheoremstyle{thm}{}{}{\slshape}{}{\scshape}{.}{0.5em}{}

\newtheoremstyle{def}{}{}{}{}{\scshape}{.}{0.5em}{}

\newtheoremstyle{rmk}{}{}{}{}{\scshape}{.}{0.5em}{}

\newtheoremstyle{claim}{}{}{}{}{\slshape}{.}{0.5em}{}

\theoremstyle{thm}

\newtheorem{theorem}{Theorem}

\newtheorem{proposition}[theorem]{Proposition}

\newtheorem*{conjecture*}{Conjecture}
\newtheorem*{theorem*}{Theorem}
\theoremstyle{def}
\newtheorem{definition}[theorem]{Definition}

\newcommand{\K}{K\"{a}hler}

\newtheorem{remark}[theorem]{Remark}

\title[$\eta$-Einstein Sasakian immersions in non-compact Sasakian space forms]{$\eta$-Einstein Sasakian immersions in non-compact Sasakian space forms}

\author{Gianluca Bande}
\address{Gianluca Bande, Dipartimento di Matematica e Informatica \\
         U\-ni\-ver\-si\-t\`a di Cagliari, Italy.}
         \email{gbande@unica.it}
         
\author{Beniamino Cappelletti--Montano}
\address{Beniamino Cappelletti--Montano, Dipartimento di Matematica e Informatica \\
         U\-ni\-ver\-si\-t\`a di Cagliari, Italy.}
         \email{b.cappellettimontano@unica.it}

\author{Andrea Loi}
\address{Andrea Loi, Dipartimento di Matematica e Informatica \\
         Universit\`a di Cagliari, Italy.}
         \email{loi@unica.it}

\date{\today}

\thanks{The authors were supported by Prin 2015 -- Real and Complex Manifolds: Geometry, Topology and Harmonic Analysis -- Italy and by Fondazione di Sardegna (Project STAGE) and Regione Autonoma della Sardegna (Project KASBA). All the authors are  members of INdAM-GNSAGA - Gruppo Nazionale per le Strutture Algebriche, Geometriche e le loro Applicazioni.}

\begin{document}
\begin{abstract}
The aim of this paper is to study Sasakian immersions  of (non-compact)  complete regular Sasakian manifolds into the Heisenberg group and into $ \mathbb{B}^N\times \mathbb{R}$ equipped with their standard Sasakian structures. We obtain a complete classification of such  manifolds  in the $\eta$-Einstein case. 
\end{abstract}
\maketitle

\section{Introduction}

%
Sasakian geometry is considered as the odd-dimensional counterpart of K\"ahler geometry. Despite the K\"ahler case, where the study of K\"ahler immersions is well developed, due to the seminal work of Calabi \cite{Cal}  (see also \cite{LoiZedda-book} for a modern treatment and an account on the subject), in the Sasakian setting there are few results. Most of the Sasakian results are concerned with finding conditions which ensure that a Sasakian submanifold is totally geodesic or similar geometric properties (see, for instance, \cite{kenmotsu, kon73, kon76}).


In \cite{CML} the second and the third authors studied Sasakian immersions into spheres. In particular they proved the following classification result:
\begin{theorem*}[\cite{CML}]\label{mainteor3-CML}
Let $S$ be a $(2n+1)$-dimensional  compact $\eta$-Einstein Sasakian manifold. Assume that there exists a Sasakian immersion of $S$ into $\mathbb{S}^{2N+1}$. If $N = n+2$ then  $S$ is Sasaki equivalent to  $\mathbb{S}^{2n+1}$ or to the Boothby-Wang fibration over $Q_n$, where $Q_n\subset \C P^{n+1}$
is the complex quadric equipped with the restriction of the Fubini--Study form of $\C P^{n+1}$.
\end{theorem*}
Since the (Sasakian) sphere is one of the three ``models'' of Sasakian space forms, it is quite natural to study as a second step the immersions into Sasakian space forms.

In this paper we give a complete characterisation of Sasakian immersions of complete, regular, $\eta$-Einstein Sasakian manifolds into a non-compact Sasakian space form $M(N, c)$, proving the following:
\begin{theorem}\label{mainteor1}
Let $S$ be a $(2n+1)$-dimensional connected, complete,  regular $\eta$-Einstein Sasakian manifold. Suppose that there exists $p\in S$, an open neightborhood $U_p$ of $p$ and a Sasakian immersion $\phi: U_p \rightarrow M(N,c)$, where $c\leq -3$. Then $S$ is Sasaki equivalent to $M(n,c)/ \Gamma$ where $\Gamma$ is some discrete subgroup of the Sasakian-isometry group of $M(n,c)$. Moreover, if $U_p=S$ then $\Gamma=\{1\}$ and $\phi$ is, up to a Sasakian transformation of $M(N,c)$, given by 
$$
\phi(z,t)=(z, 0, t+c)
$$
\end{theorem}
Theorem \ref{mainteor1} is a strong generalisation of \cite[Theorem 3.2]{kenmotsu} which asserts that a complete, $\phi$-invariant, $\eta$-Einstein submanifold of codimension $2$ of the $(2N+1)$-dimensional Heisenberg group is necessarily a totally geodesic submanifold Sasaki-equivalent to a copy of a $(2N-1)$-dimensional Heisenberg group and similarly for totally geodesic submanifolds of $\mathbb{B}^N\times \mathbb{R}$, where $\mathbb{B}^N$ denotes the unit disc of $\C^N$ equipped with the hyperbolic metric. In fact in our result there is no restriction on the codimension and  we assume that we have a Sasakian immersion instead of a $\phi$-invariant submanifold. Moreover the immersion is not necessarily injective and is not assumed to be from the whole space but from an open neighbourhood of a point.

The general philosophy in \cite{CML} and in this paper is to take into account the transversal K\"ahler geometry of the Reeb foliation. When a regular Sasakian manifold is compact as in \cite{CML}, one can use the so-called Boothby-Wang construction \cite{Boothby-Wang}, which realises the space of leaves as a K\"ahler manifold which is the base of a principal $\s^1$-fibration. Then one translates the immersion problem into a K\"ahler immersion problem of the base spaces. 

Trying the same trick in the non-compact case is more complicated because the Boothby-Wang construction fails in general, even if the Sasakian manifold is regular. Nevertheless, the Reeb foliation has the strong property to be both a totally geodesic and a Riemannian foliation. Assuming the Sasakian manifold complete, one can appeal to the result of Reinhart \cite{R} which says that the space of leaves is the base space of a fibration, and once again translate the problem into one on K\"ahler immersions.


The paper contains two other sections. In Section 2 we recall the main definitions  and some foliation theory needed in the proof of Theorem 1 to whom Section 3 is dedicated.

\section{Preliminaries}\label{preliminari}

 A \emph{contact metric manifold} is a contact manifold $(S,\eta)$ admitting a Riemannian metric $g$ compatible with the contact structure, in the sense that, defined the $(1,1)$-tensor $\phi$ by $d\eta=2g(\cdot,\phi \cdot)$, the following conditions are fulfilled
\begin{equation}\label{contactmetric}
\phi^2 = -Id + \eta \otimes \xi, \quad g(\phi \cdot, \phi \cdot) = g - \eta \otimes \eta,
\end{equation}
where $\xi$ denotes the \emph{Reeb vector field} of the contact structure, that is the unique vector field on $S$ such that
\begin{equation*}
i_{\xi}\eta=1, \ \ i_{\xi}d\eta=0.
\end{equation*}
A contact metric manifold is said to be \emph{Sasakian}  if  the following  integrability condition is satisfied
\begin{equation}\label{normal}
N_{\phi} (X,Y):=[\phi X, \phi Y] +\phi ^2 [ X, Y]-\phi [X, \phi Y]-\phi[\phi X,  Y]= -d\eta(X,Y)\xi,
\end{equation}
for any vector fields $X$ and $Y$ on $S$. 

Two Sasakian manifolds $(S_{1},\eta_{1},g_{1})$ and $(S_{2},\eta_{2},g_{2})$ are said to be \emph{equivalent} if there exists a contactomorphism $F: S_{1} \longrightarrow S_{2}$ between them which is also an isometry, i.e.
\begin{equation}\label{equivalence1}
F^{\ast} \eta_{2} = \eta_{1}, \quad F^{\ast}g_{2} = g_{1}.
\end{equation}
One can prove that if   \eqref{equivalence1} holds  then $F$ satisfies also
\begin{equation*}
F_{{\ast}_x}\circ \phi_{1} = \phi_{2}\circ F_{{\ast}_x}, \quad F_{{\ast}_x}\xi_{{1}} = \xi_{{2}}
\end{equation*}
for any $x\in S_1$.
An isometric  contactomorphism $F: S \longrightarrow S$ from a Sasakian manifold $(S,\eta, g)$ to itself will be called a \emph{Sasakian transformation} of $(S,\eta, g)$.

It is a well-known fact \cite{bg} that the foliation defined by the Reeb vector field of a Sasakian manifold $S$ has a transversal K\"ahler structure. Using the theory of Riemannian submersions one can prove that the transverse geometry is K\"{a}hler-Einstein if and only if the Ricci tensor of $S$ satisfies the following equality
\begin{equation}\label{etaeinstein}
\textrm{Ric} = \lambda g + \nu \eta\otimes\eta
\end{equation}
for some constants $\lambda$ and $\nu$. Any Sasakian manifold satisfying \eqref{etaeinstein} is said to be \emph{$\eta$-Einstein} (see \cite{BoyerGalickiMatzeu06} for more details). 


 A remarkable property of $\eta$-Einstein Sasakian manifolds is that, contrary to Sasaki-Einstein ones, they are preserved by $\mathcal{D}_a$-homothetic deformations, that is the change of structure tensors of the form
\begin{equation}\label{Dhom}
\phi_{a}:= \phi, \quad \xi_{a}:=\frac{1}{a}\xi, \quad \eta_{a}:=a \eta, \quad g_{a}:=a g + a(a-1)\eta\otimes\eta
\end{equation}
where $a>0$. 





By a \emph{Sasakian immersion} (often called invariant submanifolds or Sasakian submanifolds in the literature) of a Sasakian manifold $(S_{1},\eta_{1},g_{1})$ into the Sasakian manifold $(S_{2},\eta_{2},g_{2})$ we mean an isometric immersion $\varphi: (S_{1},g_1)\longrightarrow (S_{2},g_2)$ that preserves the Sasakian structures, i.e. such that
\begin{gather}
\varphi^{\ast}g_{2} = g_{1}, \quad \varphi^{\ast} \eta_{2} = \eta_{1}, \label{sasakianimmersion1} \\
\varphi_{{\ast}}\xi_{{1}} = \xi_{{2}}, \quad \varphi_{{\ast}}\circ \phi_{1} = \phi_{2}\circ \varphi_{{\ast}}.\label{sasakianimmersion2}
\end{gather}
%
We refer the reader to the standard references \cite{Blair2010, bg} for a more detailed account of Riemannian contact geometry and Sasakian manifolds.
\subsection*{Sasakian space forms}

Recall that the curvature tensor of a Sasakian manifold is completely determined \cite{Blair2010} by its $\phi$\emph{-sectional} curvature, that is the sectional curvature of plane sections of the type $(X, \phi X)$, for $X$ a unit vector field orthogonal to the Reeb vector field. 

A \emph{Sasakian space form} is a connected, complete Sasakian manifold with constant $\phi$-sectional curvature. According to Tanno \cite{tan2} there are exactly three simply connected Sasakian space forms depending on the value $c$ of the $\phi$-sectional curvature: the standard Sasakian sphere up $\mathcal{D}_a$-homothetic deformation if $c > -3$, the Heisenberg group $\C^n \times \R$ if $c=-3$ and the hyperbolic Sasakian space form $\mathbb{B}^n \times \R$ if $c<-3$. Notice that each simply connected space form admits a fibration over a K\"ahler manifold and in the non-compact cases the fibration is trivial. 

We denote by $M(n,c)$  the simply connected $(2n+1)$-dimensional Sasakian space form with $\phi$-sectional curvature equal to $c$. Every connected, complete Sasakian space form is Sasakian equivalent to $M(n,c)/ \Gamma$, where $\Gamma$ is a discrete subgroup of the Sasakian transformation group of $M(n,c)$.


\subsection*{Immersions and regular foliations}
We recall some basic concepts from foliation theory (see e.g. \cite{MoerdijkMrcun, palais}). Let $M$ be a smooth manifold of dimension $n$. A foliation can be defined as a maximal foliation atlas on $M$, where a foliation atlas of codimension $q$ on $M$ is an atlas
\begin{equation*}
  \left\{\varphi_{i}: U_{i} \longrightarrow \mathbb{R}^{n}=\mathbb{R}^{p}\times\mathbb{R}^{q}\right\}_{i\in I}
\end{equation*}
of $M$ such that the change of charts diffeomorphisms $\varphi_{ij}$ locally takes the form
\begin{equation*}
\varphi_{ij}(x,y)=\left(g_{ij}(x,y),h_{ij}(y)\right) .
\end{equation*}
Each foliated chart is divided into \emph{plaques}, the connected components of
\begin{equation*}
\varphi_{i}^{-1}\left(\mathbb{R}^{p}\times\left\{y\right\}\right),
\end{equation*}
where $y\in\mathbb{R}^{q}$, and the changes of chart diffeomorphism preserve this division.

\begin{definition}
A \emph{foliated map} is a map $f: (M, \mathcal{F}) \longrightarrow (M', \mathcal{F}')$ between foliated manifolds which preserves the foliation structure, i.e. which maps  leaves of $\mathcal{F}$ into leaves of $\mathcal{F}'$.
\end{definition}

\medskip

Now, let $(M, \mathcal{F}$) and $(M', \mathcal{F}')$ be foliated manifolds and $f:M \longrightarrow M'$ be an immersion. Moreover, assume that $f$ is a foliated map. Thus
\begin{equation*}
f_{{\ast}_{x}}(L(x)) \subset L'(f(x))
\end{equation*}
for each $x\in M$, where $L=T(\mathcal F)$ and $L'=T(\mathcal F')$. In particular, it follows that $\dim(\mathcal F)\leq\dim(\mathcal F')$. The proof of the following proposition is quite standard and will be omitted:

\begin{proposition}\label{coordinates}
$(M, \mathcal{F}$) and $(M', \mathcal{F}')$ be foliated manifolds of dimension $n$ and $n'$, respectively, and $f:M \longrightarrow M'$ be a foliated immersion. Suppose that $\dim(\mathcal{F})=\dim(\mathcal{F'})$. Then for each $x\in M$ there are charts $\varphi: U \longrightarrow \mathbb{R}^{p}\times\mathbb{R}^{q}$ for $M$ about $x$ and $\varphi': U' \longrightarrow \mathbb{R}^{p}\times\mathbb{R}^{q'}$ for $M'$ about $f(x)$ such that
\begin{itemize}
  \item[(i)] $\varphi(x)=(0,\ldots,0)\in\mathbb{R}^{n}$
  \item[(ii)] $\varphi'(f(x))=(0,\ldots,0)\in\mathbb{R}^{n'}$
  \item[(iii)] $F (x_{1},\ldots,x_{n})=(x_{1},\ldots,x_{n},0,\ldots,0)$, where $F:=\varphi'\circ f \circ \varphi^{-1}$
  \item[(iv)] $L(x)=\emph{span}\left \{ \frac{\partial}{\partial x_{1}}(x),\ldots,\frac{\partial}{\partial x_{p}}(x)\right\}$
  \item[(v)] $L'(f(x))=\emph{span}\left\{\frac{\partial}{\partial x_{1}}(f(x)),\ldots,\frac{\partial}{\partial x_{p}}(f(x))\right\}$
\end{itemize}
where  $p=\dim(\mathcal F)=\dim(\mathcal F')$, $q=n-p$, $q'=n'-p$.
\end{proposition}
Let $\mathcal F$ be a foliation on a manifold $M$ and let $L$ be a leaf of $\mathcal F$. It is well known that $L$ intersects at most a countable number of plaques in a foliated chart $U$. Now we give the following definition.

\begin{definition}[\cite{palais}]
A foliation $\mathcal F$ is said to be \emph{regular} if for any $x\in M$ there exists a foliated chart $U$ containing $x$ such that every leaf of $\mathcal F$ intersects at most one plaque of $U$.
\end{definition}

The following proposition is a generalisation to the non-compact case and to immersions of \cite[Proposition 3.1]{harada73-II}:
\begin{proposition}\label{prop-regularity}
Let $(M,\mathcal{F})$ and $(M',\mathcal{F}')$ be foliated manifolds such that $\dim(\mathcal{F})=\dim(\mathcal{F}')$. If there exists a foliated immersion $f: (M,\mathcal{F}) \longrightarrow (M',\mathcal{F}')$ and $\mathcal{F}'$ is regular, then $\mathcal{F}$ is also regular.
\end{proposition}
\begin{proof}
Assume that $\mathcal{F}$ is not regular. Then there exists a point $x \in M$ and a leaf $L$ of $\mathcal{F}$ such that, for any foliated chart $U$ containing $x$, $L$ intersects more then one plaque in $U$. Let us consider the foliated charts $U$ and $U'$, respectively about $x$ and $f(x)$, satisfying the properties stated in Proposition \ref{coordinates}. Then there exist at least two plaques, say $P_1= \varphi^{-1}\left(\mathbb{R}^{p}\times\left\{\textbf{y}_1\right\}\right)$ and $P_1= \varphi^{-1}\left(\mathbb{R}^{p}\times\left\{\textbf{y}_2\right\}\right)$, such that
\begin{equation}\label{intersezione}
L\cap P_1 \neq \emptyset, \quad L\cap P_2 \neq \emptyset,
\end{equation}
where $\textbf{y}_{1},\textbf{y}_{2}\in\mathbb{R}^{q}$. Notice that, for each $i\in\left\{1,2\right\}$, $f(P_{i})$ is a plaque of $\mathcal{F}'$ in $U':=f(U)$. Indeed, using Proposition \ref{coordinates}, we have $f(P_{i})=f(\varphi^{-1}(\mathbb{R}^{p}\times \left\{\textbf{y}_i\right\}))=\varphi'^{-1}(F(\mathbb{R}^{p}\times\{\textbf{y}_{i}\}))=\varphi'^{-1}(\mathbb{R}^{p}\times \{(\textbf{y}_{i},0,\ldots,0)\})$. Now, since $f$ is a foliated map,  $L'=f(L)$ is a leaf of $\mathcal{F}'$ and from \eqref{intersezione} it follows that  $L'\cap f(P_1) \neq \emptyset$ and $L'\cap f(P_2) \neq \emptyset$. But this contradicts the regularity of $\mathcal{F'}$.
\end{proof}



\section{Classification}\label{classificazione}
In this Section we prove the main result of this paper, that is the classification of connected, regular $\eta$-Einstein Sasakian manifolds immersed into Sasakian space forms.


\begin{proof}[Proof of Theorem \ref{mainteor1}]
Let $M(N,c)$ be one of the non-compact simply connected Sasakian space forms and $\pi ': M(N,c)\rightarrow K'$ the (trivial) fibration over its K\"ahler quotient. Recall that $K'$ is either $\C^N$ with its flat K\"ahler metric or the hyperbolic K\"ahler space form $\mathbb{B}^N$.

Since $S$ is complete and regular, by \cite{R} there exists a fibration $\pi :S\rightarrow K$, whose fibers are the leaves of the Reeb foliation of $S$. By assumption $S$ is Sasakian $\eta$-Einstein and then $K$ is necessarily K\"ahler-Einstein (see \cite{IP}). The fibration $\pi :S\rightarrow K$ is a Riemannian submersion and since $S$ is complete then $K$ also is (see \cite{on1, IP}).

By assumption there exists an open neighbourhood $U_p$ of $p\in S$ and a Sasakian immersion $\varphi: U_p \rightarrow M(N,c)$. By Proposition \ref{prop-regularity} the submanifold $U_p$ is still regular. The restriction of $\pi$ to $U_p$ gives then a projection of the Sasakian $\eta$-Einstein manifold $U_p$ to the K\"ahler manifold $\pi(U_p) \subset K$.

The  Sasakian immersion  $\varphi: U_p \rightarrow M(N,c)$ covers a K\"ahler immersion $i(\varphi)$ (see \cite{CML,harada72}) making the following diagram commutative:
$$
\begin{CD}
U_p@>\varphi >> M(N,c) \\
@V\pi VV @VV\pi ' V\\
\pi (U_p)\subset K @>>i(\varphi)>K' 
\end{CD}
$$

We have proved that there exists $q=\pi(p)\in K$, an open neighbourhood $V_q=\pi (U_p)$ of $q$ and a K\"ahler immersion of $V_q$ in $K'$. By Umehara \cite{um} $V_q$ is flat or complex hyperbolic.

On the other end, by \cite{DTK} (see also \cite[Theorem 5.26]{BE}), the K\"ahler-Einstein manifold $K$ is real analytic and by \cite[Theorem 4 and Theorem 10 ]{Cal}, for every $q\in K$ there exists an open neighbourhood  $V'_q$ and a K\"ahler immersion of $V'_q$ in $K'$. Then $K$ is locally flat and \cite[Formula 1.31]{IP} implies that the $\phi$-sectional curvature of $S$ is less or equal to $-3$.

By Tanno \cite{tan2} there exists a discrete group $\Gamma$  of the Sasakian transformations of $S$ such that $S=M(n, c)/ \Gamma$ and this proves the first part of the theorem. 

For the second part of the theorem, let us suppose $U_p=S$. Reasoning as before, by completeness of $K$ and by Calabi's Rigidity Theorem \cite{Cal} for K\"ahler immersions into K\"ahler space forms (see also \cite{LoiZedda-book}) one obtains the stronger result that either $K= \C^n$ or $K=\mathbb{B}^{n}$ and the projection is just the trivial fibration because in both cases $K$ contractible. Then, since $S$ is complete, the fibres of the fibration are  diffeomorphic either to $\R$ or to $\s^1$. 

The second case cannot occur because $\varphi$ is a Sasakian immersion and then it restricts to immersions on the leaves of the Reeb foliations of $S$ and $M(N,c)$. But the leaves of $M(N,c)$ are diffeomorphic to $\R$ and if the leaves of $S$ are circles we would obtain immersions of the circle in $\R$ which is not possible.

Now it remains to prove that $\varphi$ is the standard embedding up to Sasakian transformations.

First observe that, again by Calabi's Rigidity Theorem, the immersion  $ i(\varphi) : K\rightarrow K'$ has (up to unitary transformation) the following form:
$$
i(\varphi)(z_1, \ldots , z_n)=(z_1, \ldots , z_n, 0,\ldots, 0).
$$
Because the fibrations are trivial $\varphi$ must have the following expression:
$$
\varphi(z_1, \ldots , z_n, t) = (z_1, \ldots , z_n, 0,\ldots, 0, f_N ((z_1, \ldots , z_n, t))).
$$
Since $\phi$ is a Sasakian immersion, in particular we have $\varphi^*(\eta_N)=\eta_n$,  where  $\eta_N$ and $\eta_n$ are the standard contact forms of $M(N,c)$ and $M(n, c)$ respectively. Then a direct calculation of $\varphi^*(\eta_N)=\eta_n$ yields $\frac{\partial f_N}{\partial t}=1$ and $\frac{\partial f_N}{\partial x_i}=\frac{\partial f_N}{\partial y_i}=0$ for $i=1, \ldots n$, where we put $z_j=x_j+i y_j$.
\end{proof}

%

\begin{remark}
In Theorem \ref{mainteor1} the case $U_p=S$ cannot occur if $S$ is compact because if $S$ is compact, a Sasakian immersion cannot exist for otherwise, from the regularity of a compact Sasakian manifold, we would obtain a (compact) K\"ahler quotient immersed either in $\C^N$ or in $\mathbb{B}^N$, which is impossible by the Maximum Principle.
\end{remark}

The following result is a variation of Theorem \ref{mainteor1}:
\begin{theorem}\label{mainteor2}
Let $S$ be a $(2n+1)$-dimensional connected, complete, $\eta$-Einstein Sasakian manifold. Suppose that for every $p\in S$ there exists an open neighbourhood $U_p$ of $p$ and a Sasakian immersion $\phi: U_p \rightarrow M(N,c)$, where $c\leq -3$. Then $S$ is Sasaki-equivalent to $M(n,c)/ \Gamma$ where $\Gamma$ is some discrete subgroup of the Sasakian-isometry group of $M(n,c)$. Moreover, if $U_p=S$ then $\Gamma=\{1\}$ and $\phi$ is, up to a Sasakian transformation of $M(N,c)$, given by 
$$
\phi(z,t)=(z, 0, t+c)
$$
\end{theorem}
\begin{proof}
For every point  $p\in S$ we have an immersion of some $U_p$ . After possibly shrinking the open set $U_p$ we obtain an open set where the Reeb foliation is given by a fibration over a K\"ahler base. Then we proceed exactly as in the proof of Theorem \ref{mainteor1} and we obtain that $U_p$ (and then $S$) has constant $\phi$-sectional curvature at every point. Then $S$ is Sasaki-equivalent to $M(n,c)/ \Gamma$ where $\Gamma$ is some discrete subgroup of the Sasakian-isometry group of $M(n,c)$.

If $U_p=S$ we cannot directly conclude as in Theorem \ref{mainteor1} because a priori we don't know if $M(n,c)/ \Gamma$ is regular. On the other end we are assuming the existence of an immersion of $U_p=S$ into the regular Sasakian space form $M(N,c)$ and then $S$ is regular by Proposition \ref{prop-regularity}. We can now apply Theorem \ref{mainteor1} to conclude.
\end{proof}


\end{document}